\documentclass{amsart}%
\usepackage{amsfonts}
\usepackage{amsmath}
\usepackage{amssymb}
\usepackage{graphicx}%
\providecommand{\U}[1]{\protect\rule{.1in}{.1in}}
\providecommand{\U}[1]{\protect\rule{.1in}{.1in}}
\newtheorem{theorem}{Theorem}
\theoremstyle{plain}

\newtheorem{corollary}{Corollary}

\newtheorem{example}{Example}

\newtheorem{lemma}{Lemma}

\newtheorem{proposition}{Proposition}

\numberwithin{equation}{section}
\begin{document}
\title[Chaos game]{The chaos game on an iterated function system from a topological point of view}
\author{Michael F. Barnsley, Krzysztof Le\'{s}niak}
\date{\today}

\begin{abstract}
We investigate combinatorial issues relating to the use of random orbit
approximations to the attractor of an iterated function system with the aim of
clarifying the role of the stochastic process during generation the orbit. A
Baire category counterpart of almost sure convergence is presented; and a link
between topological and probabilistic methods is observed.

\end{abstract}
\maketitle

\section{Introduction}

We prove that the chaos game, for all but a $\sigma$-porous set of orbits,
yields a set that intersects all of the fibres of an attractor $A$ of a
general iterated function system\ (IFS). The IFS may not be contractive and
may possess multiple attractors. In \cite{ChaosGame} it was shown that, in
proper metric spaces, attractors are limits of certain non-stationary
stochastic chaos games; this generalized the canonical explanation, based on
stationary stochastic processes, \cite{Elton}, of why the chaos game works to
generate attractors. Here we present different results, based primarily in
topology and category rather than in stochastic processes. Also, our results
may have implications on how data strings are analyzed, as we explain next.

An iterated function system $F=(X,f_{\sigma}:\sigma\in\Sigma)$ is a finite set
of discrete dynamical systems $f_{\sigma}:X\rightarrow X$. If $(\sigma
_{k})_{k=1}^{\infty}$ is a sequence in $\Sigma$ then the corresponding chaos
game orbit \cite[p.2 and p.91]{FractalsEver} of a point $x_{0}\in X$ is the
sequence $(x_{k})_{k=0}^{\infty}$ defined iteratively by $x_{k}=f_{\sigma_{k}%
}(x_{k-1})$ for $k=1,2,...$. The chaos game may be used (i) in computer
graphics, to render pictures of fractals and other sets \cite{Superfractals,
barnsuper2, nikiel}, and (ii) in data analysis to reveal patterns in long data
strings such as DNA base pair sequences, see for example papers that cite
\cite{jeffrey}. If the maps $f_{\sigma}$ are contractions on a complete metric
space $X$, and if the sequence $(\sigma_{k})_{k=1}^{\infty}$ is suitably
chaotic or random, then the tail of $(x_{k})_{k=0}^{\infty}$ converges to the
unique attractor of $F$. In applications to computer graphics, long finite
strings $(\sigma_{k})_{k=1}^{L}$ are used, say with $L$ $=10^{9}$. In
applications to genome analysis, if $(\sigma_{k})_{k=1}^{L}$ is a long finite
sequence, say $L=2.9\times10^{9}$ for the number of base pairs in human DNA,
and if the attractor of $F$ is a simple geometrical object such as a square,
then $(x_{k})_{k=0}^{L}$ may be plotted, yielding a "picture" of $(\sigma
_{k})_{k=1}^{L}$. Such pictures may be used to identify patterns in
$(\sigma_{k})_{k=1}^{L},$ and used, for example, to distinguish different
types of DNA, \cite{jeffrey}. In the first case (i) a stochastic process is
used to define the chaos game orbit $(x_{k})_{k=0}^{L}$ and to describe the
attractor $A$ of the IFS. In the second case (ii) a deterministic process,
specified by a given data string, is used to define the chaos game orbit
$(x_{k})_{k=0}^{L}$; how this orbit sits in the attractor, that is,
\textit{the relationship between the deterministic orbit and the stochastic
orbit}, provides the pattern or signature of the string. Notice that there are
two types of chaos game here: one describing an attractor, and the other
describing a data string. Our results suggest the feasibility of data analysis
(a) using topological concepts (b) using strongly-fibred IFSs.

The type of IFS $F$ that we consider is quite general: the only restrictions
are that the underlying metric space $X$ is complete and the functions
$f_{\sigma}:X\rightarrow X$ are continuous. In Section \ref{defsec} we define
an attractor, its basin of attraction, and chaos games. In Section
\ref{defsec} we also define the fibres of an attractor and describe how
attractors are classified according to their fibre structure. The types of
fibre structure of an attractor are minimal-fibred, strongly-fibred and
point-fibred. In contrast to the situation for a contractive IFS, as in the
classical Hutchinson theory, see \cite{hutchinson}, it is not generally
possible to associate a continuous map from the code space $\Sigma^{\infty}$
onto an attractor. Consequently, results concerning the behaviour of the chaos
game cannot be inferred from analogous results, on the code space itself, by
continuous projection onto the attractor. Nonetheless, in Section
\ref{mainideasec}, we establish Theorem \ref{Theorem1}, which says that the
tail of any disjunctive chaos game orbit, starting from any point in the basin
of an attractor, converges in the Hausdorff metric to a set $C_{\infty}$ that
is both contained in the attractor and contains a point belonging to each
fibre of the attractor. This is achieved via a sequence of lemmas, similar to
ones in \cite{ChaosGame}, but replacing stochastic sequences by disjunctive
ones, and lifting the requirement that $X$ be proper. Theorem \ref{Theorem1}
allows us to prove in Section \ref{categorysec} that the chaos game, starting
from any point in the basin of strongly-fibred attractor, yields the
attractor, except for a set of strings that is small in the sense of Baire
category; specifically Theorem \ref{porousthm} says that the set of strings
for which the chaos game does not converge to the strongly-fibred attractor is
$\sigma$-porous, which is stronger than first category. In Section
\ref{probsec}, we define the notion of a disjunctive stochastic process, which
generalizes the notion of a chain with complete connections \cite{onicescu};
then we prove, \textit{as a consequence the foregoing material}, that Theorem
\ref{stochaosgame} holds: namely, a chaos game produced by disjunctive
stochastic process converges to a strongly-fibred attractor almost surely.
Thus, we see that the stochastic version is a limiting consequence of
combinatorics and topology, as it should be.

Finally, in Section \ref{rapunzelsec} we establish Theorem \ref{rapunzelthm}
-- the Rapunzel Theorem -- which illustrates the power of the disjunctiveness
in the chaos game algorithm in the commonly occuring situation where an IFS of
homeomorphisms on a compact metric space possesses a unique point-fibred
attractor $A$ and a unique point-fibred repeller $A^{\ast}$. This situation
occurs for M\"{o}bius IFSs on the Riemann sphere \cite{V}. Basically, the
result says that if $(\sigma_{k})_{k=1}^{\infty}$ is a disjunctive sequence,
then even when the point $x_{0}$ belongs to the dual repeller $A^{\ast},$ the
"usual/typical/almost always" event is that the chaos game orbit "escapes from
the tower", the disjunctive sequence "lets down her hair" and the sequence of
points in the chaos game orbit dances out of the clutches of the dual
repeller. Why is this surprising? For a number of reasons, but mainly this:
$A^{\ast}$ is the complement of the basin of attractor $A$, so it is not true
that $\lim_{k\rightarrow\infty}F^{k}(\{x\})=A$ for $x\in A^{\ast}$, and
$A^{\ast}$ may have nonempty interior.

\section{Definitions\label{defsec}}

Throughout, let $(X,d)$ be a complete metric space with metric $d$. For $b\in
X$, $C\subset X$ we denote
\[
d(b,C):=\inf_{c\in C}d(b,c),
\]
and for $B\subset X$, $\varepsilon>0$
\[
N_{\varepsilon}B:=\{x\in X:d(x,B)<\varepsilon\}.
\]
The \textbf{Hausdorff distance} between $B,C\subset X$ is defined as
\[
h(B,C):=\inf\{r>0\,:\,B\subset N_{r}C,C\subset N_{r}B\}.
\]
Let $\mathcal{K}(X)$ denote the set of nonempty compact subsets of $X$. Then
$(\mathcal{K}(X),h)$ is also a complete metric space, and may be referred to
as a \textbf{hyperspace} (\cite{FractalsEver, Superfractals, Beer, Engelking,
HandbookMulti}).

The system $F=(X,f_{\sigma}:\sigma\in\Sigma)$, comprising a finite set of
continuous maps $f_{\sigma}:X\rightarrow X$, is called an \textbf{iterated
function system} (IFS) on $X$ \cite{barnsleydemko}. Without risk of ambiguity
we use the same notation $F$ for the IFS and the associated \textbf{Hutchinson
operator }%
\[
F:\mathcal{K}(X)\rightarrow\mathcal{K}(X)\ni B\longmapsto F(B)=\bigcup
_{\sigma\in\Sigma}f_{\sigma}(B)=\{f_{\sigma}(b):\sigma\in\Sigma,b\in B\}.
\]
This map is well-defined because the $f_{\sigma}$ are continuous and finite
unions of continuous images of compacta are compacta. Furthermore, it is a
basic fact that $F:\mathcal{K}(X)\rightarrow\mathcal{K}(X)$ is continuous; a
proof can be found in \cite{BarnsleyLesniak}. The $k$-fold composition of $F$
is written as $F^{k}$.

Following \cite{ProjectiveIFS} we say that $A\in\mathcal{K}(X)$ is an
\textbf{attractor} of the IFS $F$ on $X$ when there exists an open
neighbourhood $U(A)\supset A$ such that, in the metric space $(\mathcal{K}%
(X),h)$,
\begin{equation}
F^{k}(B)\underset{k\rightarrow\infty}{\longrightarrow}A,\text{ for
}U(A)\supset B\in\mathcal{K}(X).\label{basin}%
\end{equation}
The union $\mathcal{B}(A)$ of all open neighborhoods $U(A)$ such that
\eqref{basin} is true is called the \textbf{basin }of $A$. Since
$F:\mathcal{K}(X)\rightarrow\mathcal{K}(X)$ is continuous it follows that $A$
is an invariant set for $F$, i.e. $A=F(A)$. Clearly, $A$ is the unique fixed
point of $F$ in the basin of $A$, i.e., if $B=\bigcup_{\sigma\in\Sigma
}f_{\sigma}(B)$ and $\mathcal{B}(A)\supset B\in\mathcal{K}(X)$, then $B=A$.

The \textbf{coordinate map} $\pi:\Sigma^{\infty}\rightarrow\mathcal{K}(A)$ for
$A$ (w.r.t. $F$) is defined by
\[
\pi(\rho)=\bigcap_{K=1}^{\infty}f_{{{\rho}_{1}}}\circ....\circ f_{{{\rho}_{K}%
}}(A)=:A_{{{\rho}}}%
\]
for all ${{\rho\in}\Sigma}^{\infty}$. The set $A_{{{\rho}}}$ is called a
\textbf{fibre} of $A$. If $A_{{{\rho}}}$ is a singleton for all $\rho
\in{\Sigma}^{\infty}$, then $A$ is said to be \textbf{point-fibred}. $A$ is
\textbf{strongly-fibred} means that if $\mathcal{U}$ is an open cover of $A$
and $U\in\mathcal{U}$ then there is ${{\rho\in}}\Sigma^{\infty}$ such that
$A_{{{\rho}}}\subset U$. For reasons related to a more general notion of
"attractor", all attractors of IFSs are said to be \textbf{minimally-fibred}.
Strongly-fibred is weaker than point-fibred which is weaker than the situation
where $A$ is the attractor of a contractive IFSs. Classification of attractors
according their fibration is discussed in \cite[Chapter 4]{kieninger}.

Let $({\sigma}_{1},{\sigma}_{2},{\ldots})\in{\Sigma}^{\infty}.$ The associated
orbit of $x_{0}\in U(A)$ under $F$ is the sequence $(x_{k})_{k=0}^{\infty}$
defined by
\begin{equation}
\left\{
\begin{array}
[c]{l}%
x_{0}\in U(A),\\
x_{k}:=f_{{\sigma}_{k}}(x_{k-1}),\;k\geq1.
\end{array}
\right.  \label{randomorbit}%
\end{equation}
If $({\sigma}_{1},{\sigma}_{2},{\ldots})$ is chosen according to some
stochastic process, then $(x_{k})_{k=0}^{\infty}$ is referred to as a
\textbf{random orbit}. More generally, such orbits are referred to as
\textbf{chaos game }orbits, see \cite{FractalsEver} and \cite{peakframe}, for example.

We use the notation $f_{w}:=f_{{\sigma}_{1}}{\circ}{\ldots}{\circ}f_{{\sigma
}_{k}}$ for a finite word $w=({\sigma}_{1},{\ldots},{\sigma}_{k})\in{\Sigma
}^{k}$, so that $x_{k}=f_{w}(x_{0})$. The concatenation of two words
$u=({\upsilon}_{1},{\ldots},{\upsilon}_{m})\in{\Sigma}^{m}$ and $w=({\sigma
}_{1},{\ldots},{\sigma}_{k})\in{\Sigma}^{k}$ is $uw:=({\upsilon}_{1},{\ldots
},{\upsilon}_{m},{\sigma}_{1},{\ldots},{\sigma}_{k})\in{\Sigma}^{m+k}$. Notice
that $f_{uw}=f_{u}\circ f_{w}$. We may omit the parentheses and commas; for
example $u={\upsilon}_{1}{\ldots}{\upsilon}_{m}$.

\section{Main Idea\label{mainideasec}}

Throughout this section let $F=(X,f_{\sigma}:\sigma\in\Sigma)$ be an IFS with
attractor $A\in\mathcal{K}(X)$ and basin $\mathcal{B}(A)$. Let $(x_{k}%
)_{k=0}^{\infty}$ denote the orbit of $x_{0}$ under $F$, associated with
$({\sigma}_{1},{\sigma}_{2},{\ldots})\in{\Sigma}^{\infty}$.

The following observation lies at the heart of this investigation. It is
hidden in \cite{ChaosGame}; compare also with \cite[Theorem~12.8.2]%
{LasotaMackey}.

\begin{lemma}
\label{Wandering} Given $x_{0}\in\mathcal{B}(A)$, we have $y\in A$ if and only
if, for given $\varepsilon>0$ there exists a natural number $m$ and a word
$w=(\sigma_{m},\sigma_{m-1},{\ldots},\sigma_{1})\in{\Sigma}^{m}$ such that
$d(f_{w}(x_{0}),y)<\varepsilon.$
\end{lemma}

\begin{proof}
Suppose $y\in A$ and let $\varepsilon>0$ be given. The definition of attractor
implies that there exists an iteration $m$ such that $h(F^{m}(\{x_{0}%
\}),A)<\varepsilon,$ and in particular
\[
y\in A\subset N_{\varepsilon}\,F^{m}(\{x_{0}\}).
\]
But
\[
N_{\varepsilon}\,F^{m}(\{x_{0}\})=N_{\varepsilon}\bigcup_{w\in\Sigma^{m}}%
f_{w}(\left\{  x_{0}\right\}  )=\bigcup_{w\in\Sigma^{m}}N_{\varepsilon}%
f_{w}(\left\{  x_{0}\right\}  )\text{.}%
\]
It follows that $y\in N_{\varepsilon}f_{w}(\left\{  x_{0}\right\}  )$ for some
$w\in\Sigma^{m}$. It follows that there exists a word $w=(\sigma_{m}%
,\sigma_{m-1},{\ldots},\sigma_{1})\in{\Sigma}^{m}$ such that $d(f_{w}%
(x_{0}),y)<\varepsilon$.

Conversely, suppose $y$ is such that, given $\varepsilon>0,$ there exists a
natural number $m$ and a sequence $w=(\sigma_{m},\sigma_{m-1},{\ldots}%
,\sigma_{1})\in{\Sigma}^{m}$ with $d(f_{w}(x_{0}),y)<\varepsilon.$ It follows
that $d(y,F^{m}(\left\{  x_{0}\right\}  ))<\varepsilon$. It follows that
$y\in\lim_{m\rightarrow\infty}F^{m}(\left\{  x_{0}\right\}  )=A$.
\end{proof}

For $\sigma\in\Sigma^{\infty}$, $x_{0}\in X$, $k\in\{1,2,...\}$, define%
\[
x_{k}:=x_{k}(x_{0},\sigma):=f_{{{\sigma}_{k}}}\circ....\circ f_{{{\sigma}_{1}%
}}(x_{0})\text{.}%
\]
For all $K=0,1,2,...$ define%
\[
C_{K}:=C_{K}(x_{0},{\sigma):=\overline{\bigcup_{k=K}^{\infty}\,\{x_{k}\}}%
}\text{.}%
\]
It is straightfoward to prove that $\left\{  x_{k}\right\}  _{k=0}^{\infty}$
is totally bounded; consequently $\left\{  C_{K}\right\}  _{K=0}^{\infty}$ is
a decreasing (nested) sequence of nonempty compact sets that converges in the
Hausdorff metric to a unique nonempty compact limit
\[
C_{\infty}:=C_{\infty}(x_{0},{\sigma):=}\bigcap_{K=1}^{\infty}C_{K}\text{.}%
\]

\begin{lemma}
\label{Lemma1}If $A$ is an attractor of $F$, $\mathcal{B}(A)$ is the basin of
$A$, $x_{0}\in\mathcal{B}(A)$ and $\sigma\in\Sigma^{\infty}$, then
\[
C_{\infty}(x_{0},{\sigma)}\subset A.
\]

\end{lemma}

\begin{proof}
First, it follows from Lemma \ref{Wandering} that $a\in A$ if, and only if,
there is an infinite subsequence $\left\{  k_{l}\right\}  _{l=1}^{\infty}$ of
$\left\{  k\right\}  _{k=1}^{\infty}$ and $\rho^{(k_{l})}{\in}\Sigma^{k_{l}}$
for $l=0,1,2,...,$ such that
\[
\{f_{{\rho}_{k_{l}}^{(k_{l})}}\circ...\circ f_{{\rho}_{k_{1}}^{(k_{1})}}%
(x_{0})\}_{l=1}^{\infty}%
\]
converges to $a$, namely
\[
\lim_{l\rightarrow\infty}f_{{\rho}_{k_{l}}^{(k_{l})}}\circ...\circ f_{{\rho
}_{k_{1}}^{(k_{1})}}(\{x_{0}\})=a\text{.}%
\]

Second, note that if $c\in C_{\infty}$, then there is an infinite subsequence
$\left\{  k_{m}\right\}  _{m=1}^{\infty}$ of $\left\{  k\right\}
_{k=1}^{\infty}$ such that $\{f_{{{\sigma}_{k_{m}}}}\circ....\circ
f_{{{\sigma}_{1}}}(x_{0})\}_{m=0}^{\infty}$ converges to $c$. By the first
observation, on choosing $\rho^{(k_{l})}={{\sigma}_{k_{l}}...{\sigma}_{1}}$
for $l=1,2,...$, we obtain $c\in A$.
\end{proof}

The following lemma is perhaps suprising.

\begin{lemma}
\label{Lemma4} Let $A$ be an attractor of $F$, let $\mathcal{B}(A)$ be the
basin of $A$, let $\sigma\in\Sigma^{\infty}$ and let $x_{0}\in\mathcal{B}(A).$
We have%
\[
F\left(  C_{\infty}(x_{0},\sigma)\right)  :=\bigcup\limits_{f\in F}%
f(C_{\infty}(x_{0},\sigma))\supset C_{\infty}(x_{0},\sigma)
\]

\end{lemma}

\begin{proof}
We have%
\begin{align*}
F(C_{K}(x_{0},\sigma))  &  =\bigcup\limits_{f\in F}f(C_{K}(x_{0},\sigma))\\
&  =\bigcup\limits_{f\in F}f({\overline{\bigcup_{k=K}^{\infty}\,\{f_{{{\sigma
}_{k}}}\circ....\circ f_{{{\sigma}_{1}}}(x_{0})\}}})\\
&  =\bigcup\limits_{f\in F}{\overline{\bigcup_{k=K}^{\infty}\,\{f\circ
f_{{{\sigma}_{k}}}\circ....\circ f_{{{\sigma}_{1}}}(x_{0})\}}}\\
&  \supset C_{K+1}(x_{0},\sigma)
\end{align*}

We know that $F:\mathcal{K}(X)\rightarrow\mathcal{K}(X)$ is continuous. Taking
limits of decreasing sequences, we obtain%
\[
F(C_{\infty}(x_{0},\sigma))\supset C_{\infty}(x_{0},\sigma).
\]

\end{proof}

In summary, so far, we have that for all $x_{0}\in\mathcal{B}(A)$, for all
$\sigma\in\Sigma^{\infty},$
\[
C_{\infty}(x_{0},{\sigma)}\subset F(C_{\infty}(x_{0},{\sigma)})\subset A.
\]

\begin{lemma}
\label{Lemma2} Let $A$ be an attractor of $F$, let $\mathcal{B}(A)$ be the
basin of $A$, let $x_{0}\in\mathcal{B}(A)$, $\sigma\in\Sigma^{\infty}$, and
let $\theta_{1}\theta_{2}...\theta_{P}\in\Sigma^{P}$ for some $P\in
\{1,2,...\}$. If%
\[
{{\sigma}_{M+1}...{\sigma}_{M+P}=}\theta_{1}\theta_{2}...\theta_{P}%
\]
for infinitely many distinct positive integers $M$, then
\[
f_{\theta_{P}}\circ...\circ f_{\theta_{1}}(C_{\infty}(x_{0},{\sigma))\cap
}C_{\infty}(x_{0},{\sigma)\neq\emptyset}\text{.}%
\]

\end{lemma}

\begin{proof}
Let $P=1$. We have, for all positive integers $K$ and $L$,
\begin{align}
f_{\theta_{1}}(C_{K}(x_{0},\sigma))\cap C_{K+L}(x_{0},\sigma)  &
=f_{\theta_{1}}({\overline{\bigcup_{k=K}^{\infty}\,\{f_{{{\sigma}_{k}}}%
\circ....\circ f_{{{\sigma}_{1}}}(x_{0})\}}})\cap({\overline{\bigcup
_{k=K+L}^{\infty}\,\{f_{{{\sigma}_{k}}}\circ....\circ f_{{{\sigma}_{1}}}%
(x_{0})\}}})\label{formula}\\
&  =({\overline{\bigcup_{k=K}^{\infty}\,\{f_{\theta_{1}}\circ f_{{{\sigma}%
_{k}}}\circ....\circ f_{{{\sigma}_{1}}}(x_{0})\}}})\cap({\overline
{\bigcup_{k=K+L}^{\infty}\,\{f_{{{\sigma}_{k}}}\circ....\circ f_{{{\sigma}%
_{1}}}(x_{0})\}}})\nonumber\\
&  \supset{\overline{\bigcup_{\substack{k\in\{K+L,...\}\\s.t.{{\sigma}_{k}%
=}\theta_{1}}}^{\infty}\,\{f_{{{\sigma}_{k}}}\circ....\circ f_{{{\sigma}_{1}}%
}(x_{0})\}}}).\nonumber
\end{align}
The last expression is nonempty because ${{\sigma}_{k}=}\theta_{1}$ for
infinitely many values of $k$. It follows that $\left\{  f_{\theta_{1}}%
(C_{K})\cap C_{K+L}\right\}  _{L=1}^{\infty}$ is a decreasing sequence of
nonempty compact sets. It converges to a nonempty compact set and it converges
to $f_{\theta_{1}}(C_{K})\cap C_{\infty}$ so%
\[
f_{\theta_{1}}(C_{K})\cap C_{\infty}\neq\emptyset
\]
for all $K=1,2,...$. But now $\left\{  f_{\theta_{1}}(C_{K})\cap C_{\infty
}\right\}  _{K=1}^{\infty}$ is a decreasing sequence of nonempty sets and it
converges to
\[
f_{\theta_{1}}(C_{\infty})\cap C_{\infty}\neq\emptyset.
\]
This proves the result for the case $P=1$. For the general case, replace
$f_{\theta_{1}}$ by $f_{\theta_{P}}\circ...\circ f_{\theta_{1}}$ and adjust
the expressions in (\ref{formula}) accordingly.
\end{proof}

We say that the infinite word $\sigma=({\sigma}_{1},{\sigma}_{2},{\ldots}%
)\in{\Sigma}^{\infty}$ is \textbf{disjunctive} (\cite{DisjunctiveSeq,
Staiger}) if it contains all possible finite words i.e.
\[
\forall_{m}\;\forall_{w\in{\Sigma}^{m}}\;\exists_{j}\;\forall_{l=1,{\ldots}%
,m}\;\;\sigma_{(j-1)+l}=w_{l}.
\]

In fact any finite word appears in a disjunctive sequence of symbols
infinitely often, because it reappears as part of longer and longer words.

\begin{proposition}
The sequence $(\sigma_{n})_{n=1}^{\infty}\in{\Sigma}^{\infty}$ is disjunctive
if and only if
\begin{equation}
\forall_{n,m}\;\forall_{({\tau}_{1},{\tau}_{2},{\ldots},{\tau}_{m})\in{\Sigma
}^{m}}\;\exists_{k\geq n}\;\forall_{l=1,{\ldots},m}\;\;{\tau}_{l}={\sigma
}_{k+l}. \label{complexword}%
\end{equation}

\end{proposition}

\begin{example}
\label{champernowne}(Champernowne sequence). Let us write down finite words
over the alphabet $\Sigma$: first the one-letter words, second two-letter
words etc. An infinite word made by concatenating this list creates a
disjunctive sequence of symbols in ${\Sigma}^{\infty}$, a \textbf{Champernowne
sequence}. Note that all normal sequences are disjunctive but the converse is
not true.
\end{example}

Applications of disjunctive sequences in complexity, automata theory and
number theory are described in the papers cited in \cite{DisjunctiveSeq}.

What does disjunctiveness give us? Let $\mathcal{S}_{F}$ denote the semigroup
of continuous functions from $X$ to itself, generated by $F$. That is%
\[
\mathcal{S}_{F}:=\{f_{\sigma_{1}}\circ...\circ f_{\sigma_{k}}:k\in\left\{
1,2,...\right\}  ,\sigma_{1}...\sigma_{k}\in\Sigma^{k}\}
\]
where the semigroup operation is function composition.

\begin{lemma}
\label{Lemma3}Let $A$ be an attractor of $F$, let $\mathcal{B}(A)$ be the
basin of $A$, let $x_{0}\in B,$ and let $\sigma\in\Sigma^{\infty}$ be
disjunctive. If $f\in\mathcal{S}_{F}$, then
\[
f(C_{\infty}(x_{0},\sigma))\cap C_{\infty}(x_{0},\sigma)\neq\emptyset.
\]

\end{lemma}

\begin{proof}
This is an immmediate consequence of Lemma \ref{Lemma2} combined with
disjunctiveness of $\sigma$.
\end{proof}

\begin{theorem}
\label{Theorem1}Let $A$ be an attractor of $F$, let $\mathcal{B}(A)$ be the
basin of $A$, let $x_{0}\in\mathcal{B}(A)$ and let $\sigma\in\Sigma^{\infty}$
be disjunctive. The set  $C_{\infty}(x_{0},{\sigma)}$ intersects every fibre
of $A$; that is,
\[
A_{{{\rho}}}\cap C_{\infty}(x_{0},{\sigma)\neq\emptyset}%
\]
for all ${{\rho\in}}\Sigma^{\infty}$.
\end{theorem}

\begin{proof}
We have
\begin{align*}
C_{\infty}\cap A_{{{\rho}}} &  =C_{\infty}\cap\lim_{K\rightarrow\infty
}f_{{{\rho}_{1}}}\circ....\circ f_{{{\rho}_{K}}}(A)\\
&  =C_{\infty}\cap\bigcap\limits_{K=1}^{\infty}f_{{{\rho}_{1}}}\circ....\circ
f_{{{\rho}_{K}}}(A)\text{ because decreasing,}\\
&  =\bigcap\limits_{K=1}^{\infty}(C_{\infty}\cap f_{{{\rho}_{1}}}%
\circ....\circ f_{{{\rho}_{K}}}(A))\text{ easily checked,}%
\end{align*}
But, since $C_{\infty}\subset A$ by Lemma \ref{Lemma1}, we have
\[
C_{\infty}\cap f_{{{\rho}_{1}}}\circ....\circ f_{{{\rho}_{K}}}(A)\supset
C_{\infty}\cap f_{{{\rho}_{1}}}\circ....\circ f_{{{\rho}_{K}}}(C_{\infty})
\]
for all $K$. Also, by Lemma \ref{Lemma4} and the assumption that $\sigma$ is
disjunctive, we have%
\[
C_{\infty}\cap f_{{{\rho}_{1}}}\circ....\circ f_{{{\rho}_{K}}}(C_{\infty}%
)\neq\emptyset
\]
for all $K$. It follows that $\left\{  C_{\infty}\cap f_{{{\rho}_{1}}}%
\circ....\circ f_{{{\rho}_{K}}}(A)\right\}  $ is a decreasing (nested)
sequence of non-empty compact sets. It follows that
\[
C_{\infty}\cap A_{{{\rho}}}=C_{\infty}\cap\bigcap\limits_{K=1}^{\infty
}f_{{{\rho}_{1}}}\circ....\circ f_{{{\rho}_{K}}}(A)\neq\emptyset\text{.}%
\]

\end{proof}

This says that, given any fibre $A_{\rho}$ of an attractor, there exists $p\in
A_{\rho}$ and a subsequence of $\left\{  x_{k}\right\}  $ that converges to
$p$.

\begin{corollary}
\label{Corollary1}Let $A$ be an attractor of $F$, let $B$ be the basin of $A$,
let $x_{0}\in\mathcal{B}(A)$ and let $\sigma\in\Sigma^{\infty}$ be
disjunctive. If $A$ is strongly-fibred, then
\[
C_{\infty}(x_{0},{\sigma)=A.}%
\]
That is, the tails of the random orbit
\begin{equation}
\{x_{n}\,:\,n\geq p\}\underset{p\rightarrow\infty}{\longrightarrow
}A\label{tailconverg}%
\end{equation}
converge to the attractor with respect to the Hausdorff distance, and
\begin{equation}
A=\bigcap_{p=1}^{\infty}{\overline{\bigcup_{n=p}^{\infty}\,\{x_{n}\}}%
}.\label{fillattractor}%
\end{equation}

\end{corollary}

\begin{proof}
Let $\mathcal{U}$ be a cover by balls of radius epsilon. Since $A$ is
strongly-fibred, for each $U\in\mathcal{U}$ there is $\rho\in\Sigma^{\infty}$
such that $A_{\rho}\subset U$. Hence, a point of $C_{\infty}(x_{0},{\sigma)}$
lies in each $U\in\mathcal{U}$, by Theorem \ref{Theorem1}. It readily follows
that $C_{\infty}(x_{0},{\sigma)\supset A}$. But $C_{\infty}(x_{0}%
,{\sigma)\subset A}$; hence $C_{\infty}(x_{0},{\sigma)=A}$.
\end{proof}

Note that Theorem \ref{Theorem1} is stronger than Corollary \ref{Corollary1}.

Here we digress slightly from our main themes to reflect on the name "chaos
game", since the process underlying the chaos game algorithm can be purely
deterministic and does not need to be related in any way to ergodicity (e.g.
Example~\ref{ergodicisnotenough}). In dynamical systems theory the "furthest
island" of stability is usually considered to be almost periodic behaviour,
after stationary, periodic and quasi-periodic; beyond quasi-periodicity is the
"ocean" of chaos. Following \cite{AlmostPeriodic} we recall that an infinite
sequence of symbols $\varsigma$ is \textbf{almost periodic} (or {uniformly
recursive}) if, given any finite word $\tau$ that occurs in $\varsigma$
infinitely often we can associate a positive integer $m$ such that any segment
in $\varsigma$ of length $m$ contains $\tau$ as a substring. Obviously a
disjunctive sequence cannot be almost periodic. Therefore the descriptive term
"chaos game" retains its interpretation.

\section{\label{categorysec}Categorial analysis}

Subset $\Psi\subset M$ of a metric space $M$ is called \textbf{porous} when%
\begin{equation}
\exists_{0<{\lambda}^{\prime}<1}\;\exists_{r_{0}>0}\;\forall_{\psi\in\Psi
}\;\forall_{0<r<r_{0}}\;\exists_{\upsilon\in M}\;\;N_{{\lambda}^{\prime}%
r}\{\upsilon\}\subset N_{r}\{\psi\}\setminus\Psi. \label{porous}%
\end{equation}
A countable union of porous sets is said to be $\sigma$\textbf{-porous}. A
subset of a $\sigma$-porous set is $\sigma$-porous.

Note that every $\sigma$-porous set is of the first Baire category and that
this is a proper inclusion. Moreover every $\sigma$-porous subset of euclidean
space has null Lebesgue measure. In general metric spaces one can also relate
the ideal of porous sets to the ideal of null sets under suitable assumptions.
We quote such a result next and then show its natural application in
Example~\ref{BarnVince}.

\begin{theorem}
[\cite{PorosityMeas} Propositions 3.5 \& 3.3]\label{PorousIsNull} Let $\mu$ be
the completion of a Borel regular probability measure on a separable metric
space $M$ which satisfies the doubling condition
\begin{equation}
\exists_{r_{0},c>0}\;\forall_{\psi\in M}\;\forall_{0<r<r_{0}}\;\;\mu
(N_{2r}\{\psi\})\leq c\cdot\mu(N_{r}\{\psi\}). \label{doublingmeas}%
\end{equation}
If $\Psi\subset M$ is $\sigma$-porous set, then it is null $\mu(\Psi)=0$.
\end{theorem}

We remark that the regularity assumption is superfluous since probabilistic
Borel measures on metric spaces are always regular (\cite{Billingsley} Theorem
1.1) and completion adds only subsets of null sets. Fulfilling doubling
condition everywhere implies that the measure is strictly positive (i.e.,
nonempty open sets are have positive measure); thus the support of the measure
is the whole space.

More on porosity can be found in \cite{Zajicek, PorosityMeas}. The book
\cite{Lucchetti} uses porosity to study generics in optimization problems (cf.
\cite{deBlasiMyjakPapini}). Results relating to porosity in fractal geometry
and analysis can be found for example in \cite{Chousionis,GraphDirectedMarkov}.

The following criterion will be useful.

\begin{proposition}
If $\Psi\subset M$ satisfies
\begin{equation}
\exists_{0<{\lambda}<1}\;\forall_{\psi\in\Psi}\;\forall_{n\geq1}%
\;\exists_{\upsilon\in M}\;\;N_{{\lambda}\cdot2^{-n}}\{\upsilon\}\subset
N_{2^{-n}}\{\psi\}\setminus\Psi, \label{binaryporous}%
\end{equation}
then $\Psi$ is porous.
\end{proposition}

\begin{proof}
Choose $r_{0}:=1$ and associate with $0<r<r_{0}$ the number $n\geq1$ in such a
way that
\[
2^{-n}<r\leq2\cdot2^{-n}.
\]
(Namely $n:=\mbox{entier}[\log_{2}(r^{-1})]+1$).

From \eqref{binaryporous} there exist appropriate $0<\lambda<1$ and
$\upsilon\in M$. Scale ${\lambda}^{\prime}:=\frac{\lambda}{2}$ verifies
\eqref{porous}:
\[
N_{{\lambda}^{\prime}r}\{\upsilon\}\subset N_{{\lambda}2^{-n}}\{\upsilon
\}\subset N_{2^{-n}}\{\psi\}\setminus\Psi\subset N_{r}\{\psi\}\setminus\Psi.
\]

\end{proof}

Now we recall that the \textbf{Cantor space} $({\Sigma}^{\infty},\varrho)$ is
the set of infinite words over alphabet $\Sigma$ equipped with the
\textbf{Baire metric}
\[
\varrho\left(  ({\sigma}_{i})_{i=1}^{\infty},({\upsilon}_{i})_{i=1}^{\infty
}\right)  :=2^{-\,\min\{i:{\sigma}_{i}\neq{\upsilon}_{i}\}}%
\]
for $({\sigma}_{i})_{i=1}^{\infty},({\upsilon}_{i})_{i=1}^{\infty}\in{\Sigma
}^{\infty}$ (conveniently $2^{-\,\min\emptyset}:=0$). Note that this space
$({\Sigma}^{\infty},\varrho)$ may be referred to as \textbf{code space} in
fractal geometry settings.

The topology of the Cantor space is just the Tikhonov product of the discrete
alphabet $\Sigma$ and so it is compact. But the Baire metric obeys ultrametric
triangle inequality; this provides a tree structure in the space (compare also
K\"{o}nig's lemma on trees). The Cantor space appears among others in automata
theory (e.g., \cite{DisjunctiveSeq} and references therein) and symbolic
dynamics (\cite{FractalsEver,Superfractals}).

For future reference we note that balls in the Baire metric are cylinders
\begin{equation}
\forall_{n\geq1}\;\forall_{2^{-(n+1)}<r\leq2^{-n}}\;\forall_{\psi=(\psi
_{i})_{i=1}^{\infty}\in{\Sigma}^{\infty}}\;\;N_{r}\{\psi\}=\{\psi_{1}%
\}\times{\ldots}\times\{\psi_{n}\}\times{\Sigma}^{\infty}. \label{ultraballs}%
\end{equation}

For $\tau=({\tau}_{1},{\ldots},{\tau}_{m})\in{\Sigma}^{m}$ and $p\geq1$
denote
\[
\Psi(\tau,p):=\{({\sigma}_{i})_{i=1}^{\infty}\in{\Sigma}^{\infty}%
:\exists_{k\geq p}\;\forall_{l=1,{\ldots},m}\;{\tau}_{l}={\sigma}%
_{(k-1)+l}\},
\]
the set of words that do not contain the subword $\tau$ from the $p$-th
position onwards.

\begin{lemma}
\label{nontaustrings} The set $\Psi(\tau,p)$, as a subset of the code space
$({\Sigma}^{\infty},\varrho),$ is a Borel set and porous.
\end{lemma}

\begin{proof}
To simplify notation $\Psi:=\Psi(\tau,p)$ and $\tilde{n}:=n+p$ given $n\geq1$.

Let $\psi=({\psi}_{i})_{i=1}^{\infty}\in\Psi$. We investigate $N_{2^{-n}%
}\{\psi\}\setminus\Psi$.

Define for $i\geq1$
\[
{\upsilon}_{i}:=\left\{
\begin{array}
[c]{ll}%
{\psi}_{i}, & i<\tilde{n},\\
{\tau}_{(i-\tilde{n})\mod m+1}, & i\geq\tilde{n}.
\end{array}
\right.
\]
Of course $\upsilon:=({\upsilon}_{i})_{i=1}^{\infty}\in{\Sigma}^{\infty
}\setminus\Psi$. Moreover $\upsilon\in N_{2^{-n}}\{\psi\}$, because
\[
\varrho(\upsilon,\psi)<2^{-\tilde{n}}<2^{-n}.
\]

Consider $\varsigma=({\sigma}_{i})_{i=1}^{\infty}\in{\Sigma}^{\infty}$ close
enough to $\upsilon$, namely
\[
\varrho(\varsigma,\upsilon)<2^{-(2m+p)}\cdot2^{-n}.
\]
Then ${\sigma}_{i}={\upsilon}_{i}$ for $i\leq(2m+p)+n$. So
\[
p<\tilde{n}+m<\tilde{n}+m+1<{\ldots}<\tilde{n}+m+(m-1)<2m+p+n
\]
and thus ${\sigma}_{\tilde{n}+m+l-1}={\tau}_{l}$ for $l=1,2,{\ldots},m$, which
in turn means that $\varsigma\not \in \Psi$. Additionally
\[
\varrho(\varsigma,\psi)\leq\varrho(\varsigma,\upsilon)+\varrho(\upsilon
,\psi)<2^{-1}\cdot2^{-n}+2^{-1}\cdot2^{-n}=2^{-n},
\]
which means $\varsigma\in N_{2^{-n}}\{\psi\}$. Altogether
\[
N_{{\lambda}\cdot2^{-n}}\{\upsilon\}\subset N_{2^{-n}}\{\psi\}\setminus\Psi,
\]
if we put $\lambda:=2^{-(2m+p)}$. Therefore $\Psi$ is porous subject to
condition \eqref{binaryporous}.

The complement
\begin{align*}
{\Sigma}^{\infty}\setminus\Psi &  =\bigcup_{k\geq1}{\Sigma}^{p+(k-1)}%
\times\{{\tau}_{1}\}\times{\ldots}\times\{{\tau}_{m}\}\times{\Sigma}^{\infty
}=\\
&  =\bigcup_{k\geq1}\bigcup_{{\pi}\in{\Sigma}^{p+k-1}}\,N_{2^{-(p+k-1+m)}%
}\{{\pi}\cdot{\tau}\}
\end{align*}
is a countable union of open balls due to \eqref{ultraballs}, hence $\Psi$ is Borel.
\end{proof}

\begin{theorem}
\label{GenericDisjunctive} Sequences which are not disjunctive form a Borel
$\sigma$-porous set $D^{\prime}\subset{\Sigma}^{\infty}$ w.r.t. the Baire metric.
\end{theorem}

\begin{proof}
We have
\begin{align*}
D^{\prime}  &  =\left\{  ({\sigma}_{i})_{i=1}^{\infty}\in{\Sigma}^{\infty
}:({\sigma}_{i})_{i=1}^{\infty}%
\;\mbox{does not obey condition \eqref{complexword}}\right\}  =\\
&  =\bigcup_{p\geq1}\bigcup_{m\geq1}\bigcup_{{\tau}\in{\Sigma}^{m}}\Psi({\tau
},p).
\end{align*}
Since our union is countable, it is enough to remind that the sets $\Psi
({\tau},p)$ are porous according to Lemma~\ref{nontaustrings}.
\end{proof}

We are ready to prove the main theorem of this section.

\begin{theorem}
\label{porousthm}The set of sequences $({\sigma}_{n})_{n=1}^{\infty}\in
{\Sigma}^{\infty}$, which fail to generate a random orbit that yields the
strongly-fibred attractor of the IFS $F$ via \eqref{tailconverg} and
\eqref{fillattractor} is $\sigma$-porous in $({\Sigma}^{\infty},\varrho)$.
\end{theorem}

\begin{proof}
The set of faulty sequences is a subset of $D^{\prime}$ in
Theorem~\ref{GenericDisjunctive}.
\end{proof}

\section{\label{probsec}Probabilistic analysis}

Let $Z_{n}:(S,\mathfrak{S},{\Pr})\rightarrow\Sigma$, $n=1,2,{\ldots}$, be a
sequence of random variables on a probability space $(S,\mathfrak{S},{\Pr})$,
where $\mathfrak{S}$ is a $\sigma$-algebra of events in $S$, and ${\Pr
}:{\mathfrak{S}}\rightarrow\lbrack0,1]$ probability measure. This stochastic
process generates "truely" random sequences $({\sigma}_{1},{\sigma}%
_{2},{\ldots})\in{\Sigma}^{\infty}$ i.e. ${\sigma}_{n}=Z_{n}(s)$ if the event
$s\in S$ happens at the $n$-th stage.

We define the stochastic process $(Z_{n})_{n\geq1}$ to be \textbf{disjunctive}
when
\[
\forall_{m\geq1}\;\forall_{{\tau}\in{\Sigma}^{m}}\;\;\Pr\left(  Z_{(n-1)+l}%
={\tau}_{l},\;l=1,{\ldots},m,\;\mbox{for some}\;n\right)  =1;
\]
that is, each finite word appears in the outcome with probability $1$.

In fact all words almost surely appear infinitely often. But an even stronger
assertion is true.

\begin{proposition}
\label{aslibrary} A disjunctive stochastic process $(Z_{n})_{n\geq1}$ with
values in $\Sigma$ generates a disjunctive sequence $({\sigma}_{n}%
)_{n=1}^{\infty}\in{\Sigma}^{\infty}$ as its outcome with probability $1$.
\end{proposition}

\begin{proof}
Denote for $u\in{\Sigma}^{m}$
\[
E(u):=\{(Z_{(n-1)+1},{\ldots},Z_{(n-1)+m})=u\;\mbox{for some}\;n\}.
\]
Define inductively $\gamma(p)$ to be the finite Champernowne word
(Example~\ref{champernowne}) consisting of all finite words over $\Sigma$ with
length at most $p\geq1$, and such that $\gamma(p+1)$ is just $\gamma(p)$ with
attached at its end all finite words of length $p+1$. Thus the sequence of
events $E(\gamma(p))$, $p=1,2,{\ldots}$, is descending. Moreover by
disjunctiveness of the process $\Pr(E(\gamma(p)))=1$, so
\[
\Pr\left(  \bigcap_{p\geq1}E(\gamma(p))\right)  =1.
\]
The event
\[
\bigcap_{m\geq1}\,\bigcap_{u\in{\Sigma}^{m}}\,E(u)
\]
describes the appearance of a disjunctive sequence as an outcome. Its
probability equals $1$, because
\[
\bigcap_{m\leq p}\,\bigcap_{u\in{\Sigma}^{m}}\,E(u)\supset E(\gamma(p)).
\]

\end{proof}

\begin{example}
\label{BernoulliScheme}(Bernoulli scheme; \cite{FractalsEver}). Suppose
$(Z_{n})_{n\geq1}$ is the sequence of independent random variables is
distributed according to
\begin{equation}
\exists_{\alpha>0}\;\forall_{n\geq1}\;\forall_{\sigma\in\Sigma}\;\;\Pr
(Z_{n}=\sigma)\geq\alpha. \label{positoss}%
\end{equation}
An example is the classical Bernoulli scheme with outcomes in $\Sigma$. Then
$(Z_{n})_{n\geq1}$ is disjunctive process. This follows from the
Borel-Cantelli lemma (e.g. the classic Example on p.37 after Theorem 2.2.3 in
\cite{Knill}).

For Bernoulli scheme one could alternatively apply Theorem 2.3 (item 6) from
\cite{DisjunctiveSeq} which says that the set of nondisjunctive sequences is
null with respect to the Bernoulli product measure. This follows as corollary
from combination of Theorems~\ref{GenericDisjunctive} and \ref{PorousIsNull}.
See Example~\ref{BarnVince} below for a more general case.
\end{example}

Although ergodic stochastic processes are useful in engineering applications
(e.g. \cite{GrayErg, GrayInf, Shields}) they might be too weak for reliable
simulations in probabilistic algorithms like the chaos game. (In particular,
pseudorandom number generators that pass a battery of statistical tests may
fail to generate an attractor).

\begin{example}
\label{ergodicisnotenough}(Ergodicity is not enough; \cite{Lothaire} Example
1.8.1). Let $(Z_{n})_{n\geq1}$ be the homogeneous Markov chain with states in
$\Sigma:=\{1,2\}$ such that
\begin{align*}
\forall_{{\sigma}\in{\Sigma}}\;\Pr(Z_{1}  &  =\sigma)=\frac{1}{2},\\
\forall_{n\geq2}\;\Pr(Z_{n}  &  =1\,|\,Z_{n-1}=2)=1,\\
\forall_{n\geq2}\;\forall_{{\sigma}\in{\Sigma}}\;\Pr(Z_{n}  &  =\sigma
\,|\,Z_{n-1}=1)=\frac{1}{2}.
\end{align*}
(Note that we put also condition on initial distribution of the chain). It is
ergodic (even strongly mixing as the square of its transition matrix has
positive entries; e.g.\cite{Shields} Prop.I.2.10). Moreover our chain occupies
all states almost surely:
\[
\forall_{\sigma\in\Sigma}\;\;\Pr(Z_{n}=\sigma
\;\mbox{for infinitely many}\;n)=1.
\]
Nevertheless the word "$22$" is forbidden:
\[
\Pr(Z_{n}=2,\;Z_{n+1}=2\;\mbox{for some}\;n)=0,
\]
i.e. the process lacks disjunctiveness (comp. with discussion in
\cite{Shields} chap.I.4).

In relation with Example~\ref{BarnVince} it is not hard to see that a
homogeneous finite Markov chain (with strictly positive initial distribution)
is disjunctive if and only if its transition matrix has positive entries.
\end{example}

\begin{example}
\label{BarnVince}(Chain with complete connections; \cite{ChaosGame}). Let
$(Z_{n})_{n\geq1}$ be a sequence of random variables with conditional marginal
distributions
\begin{equation}
\exists_{\alpha>0}\;\forall_{n\geq1}\;\forall_{{\sigma}_{1},{\ldots},{\sigma
}_{n}\in\Sigma}\label{conditpositoss}\\
\Pr(Z_{n}={\sigma}_{n}\,|\,Z_{n-1}={\sigma}_{n-1},{\ldots},Z_{1}={\sigma}%
_{1})\geq{\alpha},
\end{equation}
and initial distribution
\[
\Pr(Z_{1}={\sigma}_{1})\geq{\alpha}.
\]
Sometimes it is called \textbf{chain with complete connections} and
significantly generalizes usual Markov chain (\cite{Iosifescu}). We shall
indirectly prove that such minorized chains are disjunctive processes.

Define on $({\Sigma}^{\infty},\varrho)$ probabilistic measure $\mu$ to be the
completion of the joint distribution of the process $(Z_{n})_{n\geq1}$:
\[
\mu(\Xi):=\Pr((Z_{1},Z_{2},{\ldots})\in\Xi)
\]
for Borel subsets $\Xi\subset{\Sigma}^{\infty}$ (comp. \cite{Shields} Theorem
I.1.2 or \cite{GrayErg} Section 2.7 and Chapter 3). By description of balls
given in \eqref{ultraballs}
\[
\mu(N_{r}\{\varsigma\})=\Pr(Z_{1}={\sigma}_{1},{\ldots},Z_{n}={\sigma}_{n})
\]
for radii $r\in\left(  2^{-(n+1)},2^{-n}\right]  $, $n\geq1$, and centers at
$\varsigma=({\sigma}_{1},{\ldots},{\sigma}_{n},{\ldots})\in{\Sigma}^{\infty}$.

The measure $\mu$ obeys doubling condition. Indeed assume now $2r\in\left(
2^{-(n+1)},2^{-n}\right]  $ and calculate
\begin{align*}
\mu(N_{2r}\{\varsigma\})  &  =\Pr(Z_{1}={\sigma}_{1},{\ldots},Z_{n}={\sigma
}_{n})=\\
&  =\frac{\Pr(Z_{1}={\sigma}_{1},{\ldots},Z_{n+1}={\sigma}_{n+1})}{\Pr
(Z_{n+1}={\sigma}_{n+1}\,|\,Z_{n}={\sigma}_{n},{\ldots},Z_{1}={\sigma}_{1}%
)}\leq\frac{1}{\alpha}\cdot\mu(N_{r}\{\varsigma\}),
\end{align*}
where the inequality comes from \eqref{conditpositoss}.

Therefore we can apply Theorem~\ref{PorousIsNull} and
Theorem~\ref{GenericDisjunctive} to find out that nondisjunctive sequences
form Borel $\mu$-null set, so the chain generates disjunctive sequence almost surely.
\end{example}

We finalize this section by giving its main result, which follows directly
from Theorem~\ref{Theorem1} via Proposition~\ref{aslibrary}.

\begin{theorem}
\label{stochaosgame} Let $A$ be a strongly-fibred attractor of the IFS
$F=(X,f_{\sigma}:\sigma\in\Sigma)$. If the stochastic process $Z_{n}%
:(S,{\mathfrak{S}},\Pr)\rightarrow\Sigma$, $n=1,2,{\ldots}$, generating
$({\sigma}_{n})_{n=1}^{\infty}\in{\Sigma}^{\infty}$ is disjunctive, then
\eqref{tailconverg} and \eqref{fillattractor} in the statement of Corollary
\ref{Corollary1} hold with probability $1$.
\end{theorem}

\section{\label{rapunzelsec}The Rapunzel Theorem}

Let $F=(X,f_{\sigma}:\sigma\in\Sigma)$ be an IFS of homeomorphisms acting on a
compact metric space $X$. Let
\[
F^{\ast}:=(X,f_{\sigma}^{-1}:\sigma\in\Sigma)
\]
be the corresponding \textbf{dual }IFS$.$ Let $A$ be an attractor of $F$, and
let $\mathcal{B}(A)$ denote the basin of $A$. Then the set $A^{\ast
}:=X\backslash\mathcal{B}(A)$ is called the \textbf{dual repeller} and
$(A,A^{\ast})$ is called an \textbf{attractor/repeller} pair. We suppose here
that $A^{\ast}$ is an attractor of $F^{\ast}.$ It is readily proved that the
basin $\mathcal{B}(A^{\ast})$ of $A^{\ast}$ (with respect to $F^{\ast}$) is
$\mathcal{B}(A^{\ast})=X\backslash A$. Note that $\mathcal{B}(A)=X\backslash
A^{\ast}$. We furthermore suppose that $A$ is \textbf{point-fibred }with
respect to $F$ and $A^{\ast}$ is point-fibred with respect to $F^{\ast}$, see
\cite[Chapter 4]{kieninger}. This means that there exist continuous maps%
\[
\pi_{F}:{\Sigma}^{\infty}\rightarrow A\text{ and }\pi_{F^{\ast}}:{\Sigma
}^{\infty}\rightarrow A^{\ast}%
\]
that are well-defined for $\varsigma=\sigma_{1}...\sigma_{k}...\in{\Sigma
}^{\infty}$ by%
\[
\pi_{F}(\varsigma)=\lim_{k\rightarrow\infty}f_{\sigma_{1}}\circ...f_{\sigma
_{k}}(x)\text{, }x\in B,
\]%
\[
\pi_{F^{\ast}}(\varsigma)=\lim_{k\rightarrow\infty}f_{\sigma_{1}}^{-1}%
\circ...f_{\sigma_{k}}^{-1}(y)\text{, }y\in B^{\ast},
\]
where the limits are independent of $x$ and $y$. Moreover we have for all
$\varsigma\in{\Sigma}^{\infty}$
\begin{equation}
\pi_{F}(S(\varsigma))=f_{\sigma_{1}}^{-1}(\pi_{F}(\varsigma))\text{ and }%
\pi_{F^{\ast}}(S(\varsigma))=f_{\sigma_{1}}(\pi_{F^{\ast}}(\varsigma
))\label{commuteq}%
\end{equation}
where $S:{\Sigma}^{\infty}\rightarrow{\Sigma}^{\infty}$ is the shift map,
namely the continuous mapping defined by
\[
S(\varsigma)=\sigma_{2}\sigma_{3}...\text{ for all }\varsigma=\sigma_{1}%
\sigma_{2}\sigma_{3}...\in{\Sigma}^{\infty}\text{.}%
\]

In general, an IFS of homeomorphisms can have many attractor/repeller pairs.
Here we are considering only the situation where $F$ has exactly one
attractor. Our terminology and ideas derive from \cite{mcgehee1} and
\cite{mcgehee2}. However, there is a crucial difference in nomenclature,
because what McGehee calls an "attractor" we would call a "Conley attractor".

\begin{theorem}
\label{rapunzelthm} Let $F$ be an IFS of homeomorphisms with a unique
point-fibred attractor $A$ and point-fibred dual repeller $A^{\ast}$. Let
$\varsigma$ be a disjunctive sequence. Then there is a set of points
$X^{\prime}\subset X$ such that (i) $X\backslash X^{\prime}$ is $\sigma
$-porous; (ii) the chaos game orbit generated by $F,x,\varsigma$ yields $A$
for all $x\in X^{\prime}$; (iii) the dual chaos game orbit generated by
$F^{\ast},x,\varsigma$ yields $A^{\ast}$ for all $x\in X^{\prime}$.
\end{theorem}

\begin{proof}
Let $x\in X$. If $x=\pi_{F^{\ast}}(\varsigma)$ then, given any open
neighbourhood $O(x)$ of $x$ there is an open set $O(x^{\prime})\subset
O(x)\backslash\{x\}$ and, obviously, every point $y$ in $O(x^{\prime})$ either
belongs to the basin $\mathcal{B}(A)$ of $A,$ in which case its orbit yields
$A,$ or $y\in A^{\ast}$ and has a compact set of addresses $\pi_{F^{\ast}%
}^{-1}(y)$ that does not include $\varsigma$. Let $K$ be the highest index of
agreement between $\varsigma$ and any member of $\pi_{F^{\ast}}^{-1}(y)$.
Then, using equation \eqref{commuteq}, we must have
\[
f_{\sigma_{K+1}}^{-1}\circ f_{\sigma_{K}}^{-1}\circ f_{\sigma_{K-1}}^{-1}%
\circ...f_{\sigma_{1}}^{-1}(y)\in\mathcal{B}(A),
\]
(for otherwise there would have been one higher level of agreement) which
tells us (using disjunctiveness) that the chaos game orbit generated by
$F,y,\varsigma$ yields $A$. It follows that the set of points $x$, denoted
$X^{\prime\prime}$, for which the chaos game generated by $F,x,\varsigma$
yields $A$ has a $\sigma$-porous complement $X\backslash X^{\prime\prime}$.
Similarly for $x$ in the set of points $Y$, for which the chaos game generated
by $F^{\ast},x,\varsigma$ yields $A^{\ast}$ has also a $\sigma$-porous
complement. The proof is completed by choosing $X^{\prime}=X^{\prime\prime
}\cap Y.$
\end{proof}

\section*{Acknowledgement}

The authors thank Pablo Guti\'{e}rrez Barrientos for pointing out a
significant error in an earlier version of this paper. The second author
wishes to thank the Australian National University for the kind hospitality
during his fruitful stays at the Mathematical Sciences Institute in February
2011 and 2012.

\end{document}